\documentclass[12pt,a4paper]{amsart}
\usepackage{mathrsfs}
\usepackage{amssymb,amsmath,amsthm,color}
\usepackage{graphicx,mcite, cite}
\usepackage{hyperref}
\usepackage{url}
\usepackage{setspace}
\usepackage{enumerate}
\newtheorem{theorem}{Theorem}[section]
\newtheorem{lemma}[theorem]{Lemma}
\newtheorem{proof of lemma}[theorem]{Proof of Lemma}
\newtheorem{proposition}[theorem]{Proposition}

\theoremstyle{definition}
\newtheorem{definition}[theorem]{Definition}
\newtheorem{example}[theorem]{Example}

\numberwithin{equation}{section}

\makeatletter
\@namedef{subjclassname@2020}{\textup{2020} Mathematics Subject Classification}
\makeatother

\begin{document}

\title[Quaternion Weyl Transform]{Quaternion Weyl Transform and some uniqueness results}

\author{Rupak Kumar Dalai, Somnath Ghosh and R.K. Srivastava}

\address{(Rupak Kumar Dalai, R.K. Srivastava) Department of Mathematics, Indian Institute of Technology, Guwahati,
India 781039.}

\email{rupak.dalai@iitg.ac.in, rksri@iitg.ac.in}

\address{(Somnath Ghosh) Centre For Applicable Mathematics, Tata Institute of Fundamental Research, Bangalore 560065, India.}

\email{somnath.g.math@gmail.com}


\date{\today}


\begin{abstract}
In this article, we study the boundedness and several properties of the quaternion Wigner
transform. Using the quaternion Wigner transform as a tool, we define the quaternion Weyl transform (QWT)
and prove that the QWT is compact for a certain class of symbols in $L^{r}\left(\mathbb{R}^{4},
\mathbb{Q}\right)$ with $1 \leq r \leq 2.$ Moreover, it can not be extended as a bounded operator
for symbols in $L^{r}\left(\mathbb{R}^{4},\mathbb{Q}\right)$ for $2<r<\infty.$ In addition, we prove
a rank analogue of the Benedicks-Amrein-Berthier theorem for the QWT. Further, we remark about the set of
injectivity and Helgason's support theorem for the quaternion twisted spherical means.
\end{abstract}

\maketitle

\section{Introduction}
During the study of quantization problems in quantum mechanics, a type of pseudo-differential
operators was first anticipated by Hermann Weyl in \cite{We}, as an operator on $L^2(\mathbb{R}^n).$
Eventually, these operators became very useful in various areas of mathematics and physics,
especially in harmonic analysis, PDE and time-frequency analysis. Further in \cite{W}, Wong
called these operators as Weyl transform and studied its compactness as an operator on
$L^2(\mathbb{R}^n)$ for the symbol in $L^p(\mathbb{R}^{2n})$ with $1\leq p\leq 2.$
Moreover in \cite{Si}, Simon showed that for the symbol in $L^p(\mathbb{R}^{2n})$ with $2<p<\infty,$
the Weyl transform is not even bounded. In addition, for particular non-commutative groups,
e.g., Heisenberg group, quaternion Heisenberg group, upper half plane, Euclidean and Heisenberg motion
groups, bounded and unboundedness of the Weyl transform with operator valued symbol is studied by
many authors, see \cite{CZ,GS3,PZ,PZ2}.

\smallskip

In recent times, the quaternion Fourier transform (QFT) and the corresponding uncertainty principle have
received significant attention from many researchers \cite{BHHR,CKL,EF,L}. The non-commutativity of the
quaternion multiplication and the Fourier kernel make QFT different from the classical Fourier transform.
The QFT has an important application in data analysis, particularly in color image processing.
Since quaternion decomposes into two complex planes, the QFT also splits into two
Euclidean Fourier transforms, which makes QFT accessible. For instance, the version of
Hardy's theorem studied in \cite{HL} can be generalised for the QTF, see in Section \ref{sec2}.
But this approach can not enlarge for transforms concerning two functions $f$ and $g$ such as
Fourier-Wigner transform and  Weyl transform.  In this article, we define Fourier-Wigner transform
in terms of the QFT and consider its boundedness. Consequently,
we define Weyl transforms for quaternion valued symbol and study its boundedness for the symbol belongs
to $L^p(\mathbb{R}^4,\mathbb{Q}),$ the $L^p$ space of quaternion valued functions.

\smallskip

In addition, we prove some uniqueness results in this setup. In \cite{B}, Benedicks
proved that if $f\in L^1(\mathbb{R}^n),$ then both the sets $\{x\in \mathbb{R}^n:f(x)\neq0\}$ and
$\{\xi\in \mathbb{R}^n:\hat{f}(\xi)\neq0\}$ cannot have finite Lebesgue measure, unless $f=0.$ Concurrently,
in \cite{AB}, Amrein-Berthier reached to the same conclusion via the Hilbert space theory. It is extended for
certain unimodular groups in the form of the qualitative uncertainty principle (QUP). A group $G$ is said to
satisfy QUP if for each $f \in L^2(G)$ with $m\{x\in G: f(x)\neq 0\}<m(G)$ and
\begin{equation}\label{exp524}
\int_{\hat{G}}\text{rank}\hat f(\lambda)\,d\hat{m}(\lambda)<\infty
\end{equation}
implies $f=0$ (see \cite{AL}). For the Heisenberg group $\mathbb H^n,$ the condition (\ref{exp524}) of QUP
implies $\hat f$ should be supported on a set of finite Plancherel measure together with $\text{rank}\hat f(\lambda)$
is finite for almost all $\lambda.$

\smallskip

In \cite{NR}, Narayanan and Ratnakumar proved that if $f\in L^1(\mathbb{H}^n)$ is
supported on $B\times \mathbb{R},$ where $B$ is a compact subset of $\mathbb{C}^n$,
and $\hat{f}(\lambda)$ has finite rank for each $\lambda,$ then $f=0.$ Then the compact
set $B$ is replaced by finite measure set in \cite{GS,V}. An analogue result is also true
for the step two nilpotent Lie groups with MW condition \cite{CGS,GS}. Thereafter, a non trivial
extension of the result for the Heisenberg motion group is established in \cite{GS2}. In
Section \ref{sec4}, we consider a quaternion analogue of the Heisenberg group Weyl transform,
and we prove a version of the Benedicks-Amrein-Berthier result. That is, if $g\in L^1(\mathbb R^4,\mathbb Q)$
is non zero and supported on a product of two dimensional finite measure sets (defined in (\ref{exp714})),
then the Weyl transform $W(g)$ can not have finite rank. Section \ref{sec4} is concluded by a remark about
the injectivity and Helgason's support theorem for the quaternion twisted spherical means.

\section{Preliminaries and auxiliary results}\label{sec2}
\noindent\textit{The quaternion algebra $\mathbb{Q}.$}
The quaternion algebra was first originated by W. R. Hamilton in $1843.$ It is an extension of complex numbers to a four dimensional algebra, which is denoted by $\mathbb{Q}.$ Every element of $\mathbb{Q}$ is a linear combination of a real scalar and three orthogonal imaginary units $i,\, j$ and $k$ with real coefficients i.e.,
\[
\mathbb{Q}=\left\{q=q_{0}+i q_{1}+j q_{2}+k q_{3} \mid q_{0}, q_{1}, q_{2}, q_{3} \in \mathbb{R}\right\},
\]
where the imaginary units $i,\, j$ and $k$ follow Hamilton's multiplication rules
\[
i^{2}=j^{2}=k^{2}=ijk=-1, \quad i j=-j i=k, \quad j k=-k j=i, \quad k i=-i k=j.
\]
The conjugate of $q$ is defined by $\bar{q}=q_{0}-iq_{1}-jq_{2}-kq_{3},$ which satisfies
\[\bar{\bar{q}}=q, \quad \overline{p+q}=\bar{p}+\bar{q}, \quad \overline{p q}=\bar{q} \bar{p}, \quad \text{for all}\quad p, q \in \mathbb{Q}.\]
The module $|q|$ is given by
\[|q|=\sqrt{q \bar{q}}=\sqrt{\left(q_{0}^{2}+q_{1}^{2}+q_{2}^{2}+q_{3}^{2}\right)}.\]
Since $\mathbb{Q}$ is non-commutative, various results on complex field cannot directly extend to quaternions. To be in control of it, quaternion can be split into two planes spanned by $\{i-j, 1+i j\}$ and $\{i+j, 1-i j\}.$ Precisely for $q\in\mathbb{Q},$
\begin{equation}\label{exp521}
q=q_{+}+q_{-}, \quad\mbox{where } q_{\pm}=\frac{1}{2}(q \pm i q j).
\end{equation}
Explicitly in terms of the real components $q_{0},\, q_{1},\, q_{2},\, q_{3} \in \mathbb{R},$ we have
\begin{equation}\label{exp507}
q_{\pm}=\left\{q_{0} \pm q_{3}+i\left(q_{1} \mp q_{2}\right)\right\} \frac{1 \pm k}{2}=\frac{1 \pm k}{2}\left\{q_{0} \pm q_{3}+j\left(q_{1} \pm q_{2}\right)\right\}.
\end{equation}
Then it satisfy the modulus identity $|q|^2=|q_+|^2+|q_-|^2.$
The following commutator relations give a justification for the above decomposition
\begin{equation}\label{exp520}
(1+k)e^{\pm aj}=e^{\mp ai}(1+k)\mbox{ and } (1-k)e^{\pm aj}=e^{\pm ai}(1-k),\mbox{ where } a\in\mathbb{R}.
\end{equation}
A quaternion valued function
$f: \mathbb{R}^{2} \rightarrow \mathbb{Q}$ can be expressed as
\[
f\left(x_{1}, x_{2}\right)=f_{0}\left(x_{1}, x_{2}\right)+i f_{1}\left(x_{1}, x_{2}\right)+j f_{2}\left(x_{1}, x_{2}\right)+k f_{3}\left(x_{1}, x_{2}\right),
\]
where each $f_{i}\left(x_{1}, x_{2}\right)$ is a real valued function. The Schwartz space $\mathcal{S}\left(\mathbb{R}^{2}, \mathbb{Q}\right)$ be the set of smooth functions from $\mathbb{R}^{2}$ to $\mathbb{Q}$ satisfying
\[\sup _{|\alpha| \leq N} \sup _{x \in \mathbb R^{2}}\left(1+|x|^{2}\right)^{N}\left|\left(D^{\alpha} f\right)(x)\right|<\infty,\]
where $N\in\mathbb{Z}_+.$
For $1 \leq r<\infty,$ $L^{r}\left(\mathbb{R}^{2}, \mathbb{Q}\right)$ is the space of all quaternion valued functions such that
\begin{equation}
\|f\|_{r}=\left(\int_{\mathbb{R}^{2}}\left|f\left(x_{1}, x_{2}\right)\right|^{r} d x_{1} d x_{2}\right)^{1 / r}<\infty
\end{equation}
and for $r=\infty,$ $ L^{\infty}\left(\mathbb{R}^{2}, \mathbb{Q}\right)$ is the space of all essentially bounded measurable functions i.e.,
$\|f\|_{\infty}=\operatorname{ess} \sup _{x \in \mathbb{R}^{2}}|f(x)|<\infty.$

\smallskip

\noindent\textit{Quaternion Fourier transform.} For $f \in L^{1}\left(\mathbb{R}^{2}, \mathbb{Q}\right),$
the quaternion Fourier transform (QFT) $\mathcal{F}(f): \mathbb{R}^{2} \rightarrow \mathbb{Q}$ is defined as
\begin{equation}
\mathcal{F}(f)\left(y_{1}, y_{2}\right)=\int_{\mathbb{R}^{2}} e^{-2 \pi i x_{1} y_{1}} f\left(x_{1}, x_{2}\right) e^{-2 \pi j x_{2} y_{2}} d x_{1} d x_{2}.
\end{equation}
This is also known as two sided quaternion Fourier transform. Further, $f$ can be reconstructed from the QFT.
\begin{theorem}{\em \cite{CKL}(Inverse QFT)}
If $f, \mathcal{F}(f) \in L^{1}\left(\mathbb{R}^{2}, \mathbb{Q}\right),$ then
\[f(x_1,x_2)=\int_{\mathbb{R}^{2}} e^{ 2 \pi i y_{1} x_{1}} \mathcal{F}(f)(y_1,y_2) e^{ 2 \pi jy_{2} x_{2}} dy_1dy_2.\]
\end{theorem}

\begin{theorem}{\em \cite{CKL}(Plancherel theorem for QFT)}
If $f \in L^{2}\left(\mathbb{R}^{2}, \mathbb{Q}\right),$ then
\[\|f\|_{2}=\|\mathcal{F}(f)\|_{2}.\]
\end{theorem}
Due to the fact that $ie^{aj}=e^{-aj}i,\,a\in\mathbb{R},$ the QFT of $f\in L^2\left(\mathbb{R}^{2},\mathbb{Q}\right)$
can be reframed by
\begin{equation}\label{exp516}
\begin{aligned}
\mathcal{F}(f)\left(y_{1}, y_{2}\right) &=\int_{\mathbb{R}^{2}} e^{-2 \pi i x_{1} y_{1}} f\left(x_{1}, x_{2}\right) e^{-2 \pi j x_{2} y_{2}} d x_{1} d x_{2} \\
&=\int_{\mathbb{R}^{2}} e^{- 2 \pi ix_{1} y_{1}}e^{- 2 \pi jx_{2} y_{2}}\left(f_{0}\left(x_{1}, x_{2}\right)+j f_{2}\left(x_{1}, x_{2}\right)\right)  d x_{1} d x_{2} \\
&+\int_{\mathbb{R}^{2}} e^{- 2 \pi ix_{1} y_{1}}e^{ 2 \pi jx_{2} y_{2}}\left(i f_{1}\left(x_{1},x_{2}\right)+\mathrm{k} f_{3}\left(x_{1},x_{2}\right)\right)  d x_{1} d x_{2} \\
 &=\int_{\mathbb R^{2}} e^{-2 \pi i x_{1} y_{1}}e^{-2 \pi j x_{2} y_{2}} \tilde{f}^j\left(x_{1}, x_{2}\right) d x_{1} d x_{2},
\end{aligned}
\end{equation}
where
\begin{equation}
\tilde{f}^j\left(x_{1}, x_{2}\right)=f_{0}\left(x_{1}, x_{2}\right)+\mathrm{i} f_{1}\left(x_{1},-x_{2}\right)+\mathrm{j} f_{2}\left(x_{1}, x_{2}\right)+\mathrm{k} f_{3}\left(x_{1},-x_{2}\right).
\end{equation}
Similarly,
\begin{equation}\label{exp517}
\mathcal{F}(f)\left(y_{1}, y_{2}\right)=\int_{\mathbb R^{2}} \tilde{f}^i\left(x_{1}, x_{2}\right)
e^{-2 \pi i x_{1} y_{1}}e^{-2 \pi j x_{2} y_{2}}  d x_{1} d x_{2},
\end{equation}
where
\begin{equation}\label{exp506}
\tilde{f}^i\left(x_{1}, x_{2}\right)=f_{0}\left(x_{1}, x_{2}\right)+\mathrm{i} f_{1}\left(x_{1},x_{2}\right)+\mathrm{j} f_{2}\left(-x_{1}, x_{2}\right)+\mathrm{k} f_{3}\left(-x_{1},x_{2}\right).
\end{equation}
Here the right-hand side of (\ref{exp516}) is the left-sided quaternion Fourier transform of $\tilde{f}^j,$
and the right-hand side of (\ref{exp517}) is the right-sided quaternion Fourier transform of $\tilde{f}^i.$
Note that $\|f\|_{2}=\|\tilde{f}^i\|_{2}=\|\tilde{f}^j\|_{2}$.
Further, if $f(x_1,-x_2)=f(x_1,x_2),$ then $\tilde{f}^j=f$ and if $f(-x_1,x_2)=f(x_1,x_2),$ then $\tilde{f}^i=f.$

\begin{theorem}{\em (Parseval's theorem for QFT)}
Let $f, g \in L^{1}\left(\mathbb{R}^{2}, \mathbb{Q}\right),$ then
\begin{equation}\label{exp523}
\int_{\mathbb{R}^{2}}\mathcal{F}(f)(y)\overline{\mathcal{F}(g)(y)}dy
=\int_{\mathbb{R}^{2}}\tilde{f}^i(x)\overline{\tilde{g}^i(x)}dx,
\end{equation}
where $\tilde{f}^i$ and $\tilde{g}^i$ are defined as in (\ref{exp506}).
In particular, if $f$ and $g$ are even in first variable, then
$\left\langle\mathcal{F}(f),\mathcal{F}(g)\right\rangle=\left\langle f,g\right\rangle.$
\end{theorem}
\begin{proof}
The left integral $\int_{\mathbb{R}^{2}}\mathcal{F}(f)(y)
\overline{\mathcal{F}(g)(y)}dy$ of (\ref{exp523}) is equal to
\[\int_{\mathbb{R}^{2}}\int_{\mathbb{R}^{2}}\int_{\mathbb{R}^{2}}
e^{-2 \pi i x_{1} y_{1}} f\left(x_{1}, x_{2}\right) e^{-2 \pi j x_{2} y_{2}}
e^{2 \pi j x'_{2} y_{2}}\overline{ g\left(x'_{1}, x'_{2}\right) }
e^{2 \pi i x'_{1} y_{1}}d xd x'dy.\]
Adapting the method used in (\ref{exp517}), the above integral becomes
\begin{align*}
&\int_{\mathbb{R}^{2}}\int_{\mathbb{R}^{2}}\int_{\mathbb{R}^{2}} \tilde{f}^i(x)
e^{-2 \pi i x_{1} y_{1}}e^{-2 \pi j x_{2} y_{2}} e^{2 \pi j x'_{2} y_{2}}
e^{2 \pi i x'_{1} y_{1}}\overline{\tilde{g}^i(x')}d xd x'dy\\
=&\int_{\mathbb{R}^{2}}\int_{\mathbb{R}^{2}} \tilde{f}^i(x)\delta(x'-x)
\overline{\tilde{g}^i(x')} d x'd x=\int_{\mathbb{R}^{2}}\tilde{f}^i(x)
\overline{\tilde{g}^i(x)}dx.
\end{align*}
\end{proof}

As in (\ref{exp521}), $f \in L^{1}\left(\mathbb{R}^{2}, \mathbb{Q}\right)$
can be decomposed by $f=f_{+}+f_{-}.$ Then the QFT becomes
\begin{equation}\label{exp525}
\mathcal{F}(f)\left(y_{1}, y_{2}\right)=\mathcal{F}\left(f_{+}\right)\left(y_{1}, y_{2}\right)
+\mathcal{F}\left(f_{-}\right)\left(y_{1}, y_{2}\right).
\end{equation}
Using (\ref{exp520}), the QFT of $f_{\pm}$ reduces to the Euclidean Fourier transform
\begin{equation}\label{exp513}
\mathcal{F}\left(f_{\pm}\right)=\int_{\mathbb{R}^{2}} e^{-2 \pi i\left(x_{1} y_{1} \mp x_{2} y_{2}\right)} f_{\pm} d x_{1} d x_{2}.
\end{equation}
Due to the modulus identity $|q|^{2}=\left|q_{+}\right|^{2}+\left|q_{-}\right|^{2},$
we have the following two relations for $f \in L^{1}\left(\mathbb{R}^{2}, \mathbb{Q}\right),$
\begin{equation}
\begin{array}{l}
\left|f\left(x_{1}, x_{2}\right)\right|^{2}=\left|f_{+}\left(x_{1}, x_{2}\right)\right|^{2}
+\left|f_{-}\left(x_{1}, x_{2}\right)\right|^{2}, \\
\left|\mathcal{F}(f)\left(y_{1}, y_{2}\right)\right|^{2}=\left|
\mathcal{F}\left(f_{+}\right)\left(y_{1}, y_{2}\right)\right|^{2}+\left|
\mathcal{F}\left(f_{-}\right)\left(y_{1}, y_{2}\right)\right|^{2}.
\end{array}
\end{equation}
\textit{Hardy's Theorem and rotations.} In \cite{HL}, an extension of Hardy's classical
characterization of real Gaussians to the case of complex Gaussians proved for the
Euclidean Fourier transform, and generalized to complex spaces of several variables
in \cite{TP}. Making use of (\ref{exp525}), we notice that this can also be done for the
QFT. Though the result in \cite{TP} is true for arbitrary $n$-dimensional space, for our
purpose, we state it for $n=2.$

\begin{theorem}{\em \cite{TP}}\label{th502}
Let $f \in L^{1}\left(\mathbb{R}^{2}\right)$ and there is some
$\psi^{0}=\left(\psi_{1}^{0}, \psi_{2}^{0}\right) \in (-\pi / 2, \pi / 2)^{2}$
such that the integral
\[
\widehat{f}\left(e^{i \psi^{0}} s\right)=\int_{\mathbb{R}^{2}} f(x)
\exp \left(-2 \pi i x\left(e^{i \psi^{0}} s\right)\right) d x
\]
converges for all $s=\left(s_{1}, s_{2}\right) \in \mathbb{R}^{2}$ and satisfies
\begin{equation}\label{exp526}
\left|\widehat{f}\left(e^{i \psi^{0}} s\right)\right| \leq C_{1} e^{-\pi|s|^{2} / \alpha},
\end{equation}
where $C_{1}$ and $\alpha$ are positive constants. Then $f$ has an analytic
extension to $\mathbb{C}^{2} .$ Furthermore, suppose
$\theta^{0}=\left(\theta_{1}^{0}, \theta_{2}^{0}\right) \in \mathbb{R}^{2}$ such
that the extension of $f$ satisfies
\begin{equation}\label{exp527}
\left|f\left(e^{i \theta^{0}} r\right)\right| \leq C_{2} e^{-\pi \alpha|r|^{2}}
\end{equation}
for some $C_{2}>0$ and all $r=\left(r_{1}, r_{2}\right) \in \mathbb{R}^{2}$, where $\alpha$
is as above. Then $f$ is a rotation of a multiple of $e^{-\pi \alpha x^{2}}$
through the angle $-\theta_{j}^{0}$ with respect to $x_{j}$ in the $z_{j}$ plane $(j=1,2):$
\[
f(z)=C \exp \left(-\pi \alpha\left(e^{-2 i \theta_{1}^{0}} z_{1}^{2}+e^{-2 i \theta_{2}^{0}}
z_{2}^{2}\right)\right)=C \exp \left(-\pi \alpha\left(e^{-i \theta^{0}} z\right)^{2}\right).
\]
Moreover, we have
$-\theta_{j}^{0} \equiv \psi_{j}^{0} \bmod \pi, \quad\left|\psi_{j}^{0}\right|<\frac{\pi}{4},
\quad j=1,2.$
\end{theorem}
Consider the function $g \in L^{1}\left(\mathbb{R}^{2}, \mathbb{Q}\right),$ then
using (\ref{exp507}) $g$ can be decomposed as
$g=\tilde{g}_{+}\frac{1+k}{2}+\tilde{g}_{-}\frac{1-k}{2},$
where $\tilde{g}_{\pm}$ are complex valued functions. Since $\mathcal{F}(g)$
splits into $\mathcal{F}(\tilde{g}_{\pm})\frac{1 \pm k}{2}$ in terms of the
Euclidean Fourier transform, if $\mathcal{F}(g)$ satisfies (\ref{exp526}) then $\tilde{g}_{\pm}$
will also satisfy (\ref{exp526}). Hence both $\tilde{g}_{\pm}$ have analytic extension to $\mathbb{C}^{2} .$
Furthermore, if $g$ satisfies (\ref{exp527}) then $\tilde{g}_{\pm}$ follow the above theorem and conclude with
rotations of a multiple of $e^{-\pi \alpha x^{2}}$ through the angle $-\theta^{0}.$
Thus, we have the following version of Hardy's theorem with regard to the QFT.

\smallskip

{\em Suppose $g \in L^{1}\left(\mathbb{R}^{2}, \mathbb{Q}\right)$ such that $\mathcal{F}(g)(e^{i\psi^{0}}s)$
converges for all $s,$ and satisfies $|\mathcal{F}(g)(e^{i \psi^{0}} s)| \leq C_{1} e^{-\pi|s|^{2} / \alpha},$
where $\psi^0, C_1$ and $\alpha$ are as in theorem \ref{th502}. Then $g$ has an analytic extension to $\mathbb{C}^{2}.$ Furthermore, suppose $\theta^{0}\in \mathbb{R}^{2}$ such that the extension of $g$ satisfies
$|g(e^{i \theta^{0}} r)|\leq C_{2} e^{-\pi \alpha|r|^{2}}$ for some $C_{2}>0$ and all $r\in \mathbb{R}^{2}.$
Then $g$ is a rotation of a multiple of $e^{-\pi \alpha x^{2}}$ through the angle $-\theta^{0}$ on $\mathbb{C}^2.$
That is
\[g(z)=C \exp \left(-\pi \alpha\left(e^{-i \theta^{0}} z\right)^{2}\right)
+C' \exp \left(-\pi \alpha\left(e^{-i \theta^{0}} z\right)^{2}\right)k.\]
Moreover, we have $-\theta_{j}^{0} \equiv \psi_{j}^{0} \bmod \pi, \quad\left|\psi_{j}^{0}\right|<\frac{\pi}{4}, \quad j=1,2.$}

\section{The Fourier-Wigner Transform and Weyl transform}
\subsection{Boundedness of Fourier-Wigner Transform.}
Before exploring the Weyl transform, define a related transform known as Fourier-Wigner transform, a helpful tool for studying the Weyl transform.
\begin{definition}
Let $f$ and $g$ be in $\mathcal{S}\left(\mathbb{R}^{2},\mathbb{Q}\right).$ Then the Fourier-Wigner transform of $f$ and $g$ defined by
\begin{equation}\label{exp514}
V(f, g)(q, p)=\int_{\mathbb{R}^{2}} e^{2 \pi i\left(q_{1} x_{1}+\frac{1}{2} q_{1} p_{1}\right)}{f(x+p)} \overline{g(x)} e^{2 \pi j\left(q_2 x_{2}+\frac{1}{2} q_{2} p_{2}\right)} d x,
\end{equation}
where $q=(q_1,\,q_2),$ and $p=(p_1,\,p_2)$ in $\mathbb{R}^2.$
\end{definition}
Using a simple change of variable we can rewrite (\ref{exp514}) as
\begin{equation}\label{exp510}
V(f, g)(q, p)=\int_{\mathbb{R}^{2}} e^{2 \pi i q_{1} y_{1}} f\left(y+\frac{p}{2}\right) \overline{g\left(y-\frac{p}{2}\right)} e^{2 \pi j q_{2} y_2} d y.
\end{equation}
Note that $V: \mathcal{S}\left(\mathbb{R}^{2}, \mathbb{Q}\right) \times \mathcal{S}\left(\mathbb{R}^{2},
\mathbb{Q}\right) \rightarrow \mathcal{S}\left(\mathbb{R}^{4}, \mathbb{Q}\right)$ is a bilinear map.

Now we regard some properties of Wigner transform. For studying the Weyl transform, we requird the notion of the Wigner transform of two functions from $L^{2}\left(\mathbb{R}^{2},\mathbb{Q}\right).$ Consequently, we begin with the QFT of the Fourier-Wigner transform.
\begin{theorem}
Let $f$ and $g$ be in $\mathcal{S}\left(\mathbb{R}^{2},\mathbb{Q}\right)$, then for $x,\,\xi\in\mathbb{R}^2,$
\begin{equation}\label{exp522}
\mathcal{F}\left(V(f, g)\right)(x, \xi)=\int_{\mathbb{R}^{2}} e^{-2\pi i \xi_1 p_1} f\left(x+\frac{p}{2}\right)
\overline{g\left(x-\frac{p}{2}\right)} e^{-2\pi j \xi_2 p_2}d p.
\end{equation}

\end{theorem}
\begin{proof}
For $\varepsilon>0,$ define the function $W_{\varepsilon}$ on $\mathbb{R}^{4}$ by
\begin{equation}\label{exp505}
W_{\varepsilon}(x, \xi)=\int_{\mathbb{R}^{2}} \int_{\mathbb{R}^{2}} e^{-\varepsilon^{2}\pi|q|^{2}} e^{-2\pi i x_1q_1-2\pi i \xi_1p_1} V(f, g)(q, p)e^{-2\pi j x_2q_2-2\pi j \xi_2p_2} d q d p.
\end{equation}
By using Fubini's theorem and the fact that Euclidean Fourier transform of
\begin{equation}\label{exp502}
\varphi(x)=e^{-\pi |x|^2} ~\mbox{ for } x \in \mathbb{R}^{2}
\end{equation}
is equal to $\varphi,$ we get
\begin{align*}
W_{\varepsilon}(x, \xi)=&\int_{\mathbb{R}^{2}} e^{-2\pi i \xi_1p_1}
\int_{\mathbb{R}^{2}}\left(\int_{\mathbb{R}}e^{-\varepsilon^{2}\pi |q_1|^{2}}e^{-2\pi i(x_1-y_1)q_1}d q_1\right) f\left(y+\frac{p}{2}\right) \overline{g\left(y-\frac{p}{2}\right)}\\
&\qquad\quad\left(\int_{\mathbb{R}} e^{-2\pi j(x_2-y_2)q_2} e^{-\varepsilon^{2}\pi |q_2|^{2}} d q_2\right)dy\,e^{-2\pi j \xi_2p_2} d p \\
=&\int_{\mathbb{R}^{2}} e^{-2\pi i \xi_1p_1}\left(\int_{\mathbb{R}^{2}} \varepsilon^{-2} e^{\frac{\pi|x-y|^{2}}{-\varepsilon^{2}}} f\left(y+\frac{p}{2}\right) \overline{g\left(y-\frac{p}{2}\right)} d y\right)e^{-2\pi j \xi_2p_2} d p.\\
\end{align*}
Now, for each $p\in\mathbb{R}^{2},$ define the function $F_{p}$ by
\begin{equation}\label{exp501}
F_{p}(y)=f\left(y+\frac{p}{2}\right) \overline{g\left(y-\frac{p}{2}\right)}.
\end{equation}
Using (\ref{exp501}), we get
\[W_{\varepsilon}(x, \xi)=\int_{\mathbb{R}^{2}} e^{-2\pi i \xi_1p_1}\left(\varphi_{\varepsilon}\ast F_p\right)(x) e^{-2\pi j \xi_2p_2} d p,\]
where
$\varphi_{\varepsilon}(x)=\varepsilon^{-2} \varphi\left(\frac{x}{\varepsilon}\right).$
For each fixed $p$ in $\mathbb{R}^{2},$ the equation (\ref{exp501}) leads
\[\varphi_{\varepsilon}\ast F_p \rightarrow\left(\int_{\mathbb{R}^{2}} \varphi(x) d x\right) F_{p}=F_{p}\]
uniformly on compact subsets of $\mathbb{R}^{2}$ as $\varepsilon \rightarrow 0$. Let $N$ be any positive integer. Then there exists a positive constant $C_{N}$ such that
\begin{equation}\label{exp504}
\begin{aligned}
\left|\left(\varphi_{\varepsilon}\ast F_p\right)(x)\right| \leq\left\|\varphi_{\varepsilon}\right\|_1\left\|F_{p}\right\|_\infty&
\leq\sup _{y \in \mathbb{R}^{2}}\left|f\left(y+\frac{p}{2}\right) g\left(y-\frac{p}{2}\right)\right|\\
&\leq C_{N}\left(1+|p|^{2}\right)^{-N},
\end{aligned}
\end{equation}
for all $\varepsilon>0 .$ So, by inequality (\ref{exp504}) and the Lebesgue dominated convergence theorem, we get
\[
\lim _{\varepsilon \rightarrow 0} W_{\varepsilon}(x, \xi)=\int_{\mathbb{R}^{2}} e^{-2\pi i \xi_1p_1} f\left(x+\frac{p}{2}\right) \overline{g\left(x-\frac{p}{2}\right)}e^{-2\pi i \xi_2p_2} d p.
\]
But, applying the Lebesgue dominated convergence theorem in (\ref{exp505}), we can conclude that
\[\begin{aligned}	
\lim _{\varepsilon \rightarrow 0} W_{\varepsilon}(x, \xi) &=\int_{\mathbb{R}^{2}} \int_{\mathbb{R}^{2}} e^{-2\pi i x_1q_1-2\pi i \xi_1p_1} V(f, g)(q, p)e^{-2\pi j x_2q_2-2\pi j \xi_2p_2} d q d p \\	
&=\mathcal{F}\left(V(f, g)\right)(f, g)(x, \xi).
\end{aligned}\]
\end{proof}
In view of the above result, the Wigner transform $W(f,g)$ of $f,\,g\in \mathcal{S}\left(\mathbb{R}^{2},\mathbb{Q}\right)$ can be defined by
\begin{equation}\label{exp509}
W(f, g)(x, \xi)=\int_{\mathbb{R}^{2}} e^{-2\pi i \xi_1 p_1} f\left(x+\frac{p}{2}\right) \overline{g\left(x-\frac{p}{2}\right)} e^{-2\pi j \xi_2 p_2}dp.
\end{equation}
Some of its properties are obtained in the following.
\begin{proposition}\label{prop502}
Let $f,\,g\in\mathcal{S}\left(\mathbb{R}^{2},\mathbb{Q}\right)$ and $2\leq r\leq\infty,$ then the Wigner transform $W(f, g)\in L^r\left(\mathbb{R}^{4},\mathbb{Q}\right).$ Moreover,
\[\|W(f, g)\|_r\leq\|f\|_2\|g\|_2.\] Hence
$W:\mathcal{S}\left(\mathbb{R}^{2},\mathbb{Q}\right) \times \mathcal{S}\left(\mathbb{R}^{2},\mathbb{Q}\right) \rightarrow L^r\left(\mathbb{R}^{4},\mathbb{Q}\right)$ can be extended uniquely to a
bilinear operator
$W: L^{2}\left(\mathbb{R}^{2},\mathbb{Q}\right) \times L^{2}\left(\mathbb{R}^{2},\mathbb{Q}\right) \rightarrow L^{r}\left(\mathbb{R}^{4},\mathbb{Q}\right).$
\end{proposition}
\begin{proof}
Let $f,\,g\in\mathcal{S}\left(\mathbb{R}^{2},\mathbb{Q}\right).$ By Cauchy-Schwartz inequality, we get
\begin{align*}
|W(f, g)(x, \xi)|&=\left|\int_{\mathbb{R}^{2}} e^{-2\pi i \xi_1 p_1} f\left(x+\frac{p}{2}\right) \overline{g\left(x-\frac{p}{2}\right)} e^{-2\pi j \xi_2 p_2}d p\right|\\
&\leq\int_{\mathbb{R}^{2}}\left| f\left(x+\frac{p}{2}\right)\right|\left| \overline{g\left(x-\frac{p}{2}\right)}\right|d p\leq\|f\|_2\|g\|_2.
\end{align*}
Hence \[\|W(f, g)\|_\infty\leq\|f\|_2\|g\|_2.\]
Now for $r=2,$ Plancherel theorem gives
{\small \[\begin{aligned}
\int_{\mathbb{R}^2}\int_{\mathbb{R}^2}\left|W(f, g)(x, \xi)\right|^{2} dxd\xi &=\int_{\mathbb{R}^{2}}\left(\int_{\mathbb{R}^{2}}\left| f\left(x\right) \overline{g\left(x-p\right)}\right|^{2} dp\right) dx
&=\|f\|_2\|g\|_2.
\end{aligned}\]}
Thus the result follows from the Riesz-Thorin interpolation Theorem.
\end{proof}
Though we show that the Wigner transform is $L^{2}$ bounded for all $f$ and $g$ in $L^{2}\left(\mathbb{R}^{2},\mathbb{Q}\right),$ we could only prove the orthogonality relation for a sub-class of functions.
\begin{theorem}{\em (The Moyal Identity)}
Let $f_{1},\, g_{1},\, f_{2},\,g_{2}\in\mathcal{S}\left(\mathbb{R}^{2},\mathbb{Q}\right)$ such that $f_i(x_1,x_2)=f_i(-x_1,x_2)$ and $g_i(x_1,x_2)=g_i(-x_1,x_2)$ for $i=1,2.$ Then
\[\left\langle W\left(f_{1}, g_{1}\right), W\left(f_{2}, g_{2}\right)\right\rangle=\left\langle f_{1},\, f_{2} \left\langle \bar{g_2},\, \bar{g_1}\right\rangle\right\rangle.\]
\end{theorem}
\begin{proof}
Consider
\begin{equation}\label{exp519}
\left\langle W\left(f_{1}, g_{1}\right), W\left(f_{2}, g_{2}\right)\right\rangle
=\int_{\mathbb{R}^2}\int_{\mathbb{R}^2}W\left(f_{1}, g_{1}\right)(x, \xi) \overline{W\left(f_{2}, g_{2}\right)(x, \xi)}dx d\xi
\end{equation}
\begin{align*}
=\int_{\mathbb{R}^{2}}\int_{\mathbb{R}^{2}}\int_{\mathbb{R}^{2}}\int_{\mathbb{R}^{2}}e^{-2\pi i \xi_1 p_1} &f_1\left(x+\frac{p}{2}\right) \overline{g_1\left(x-\frac{p}{2}\right)} e^{-2\pi j \xi_2 p_2}\\
&\times e^{2\pi j \xi_2p'_2}g_2\left(x-\frac{p'}{2}\right)\overline{f_2\left(x+\frac{p'}{2}\right)}
e^{2\pi i \xi_1 p'_1}dpdp'dxd\xi.
\end{align*}
We can write $f_1=\frac{1-k}{2}f_{11}+\frac{1+k}{2}f_{12},\,\bar{f_2}=f_{21}\frac{1-k}{2}+f_{22}\frac{1+k}{2}$ from (\ref{exp507}) and $\bar{g}_1=ig_{11}-g_{12},\,g_2=ig_{21}-g_{22},$ where $f_{lm}$ and $g_{lm}$ for $1\leq l,m\leq2$ are combination of two real valued functions with imaginary unit $j$  such as $h_1+jh_2.$
Then we have
\begin{align*}
&f_1\left(x+\frac{p}{2}\right) \overline{g_1\left(x-\frac{p}{2}\right)}=\sum_{1\leq l,m\leq 2} \frac{1+(-1)^{l}k}{2}f_{1,l}\left(x+\frac{p}{2}\right)i^{m}g_{1,m}\left(x-\frac{p}{2}\right),\\
&g_2\left(x-\frac{p'}{2}\right)\overline{f_2\left(x+\frac{p'}{2}\right)}=\sum_{1\leq l',m'\leq 2}
i^{l'}g_{2,l'}\left(x-\frac{p'}{2}\right)f_{2,m'}\left(x+\frac{p'}{2}\right)\frac{1+(-1)^{m'}k}{2}.
\end{align*}
Replacing these values on above integral and then using the commutator rule (\ref{exp520}) together with the assumptions of functions, (\ref{exp519}) can be phrased as
\begin{align*}
\int_{\mathbb{R}^{2}}\int_{\mathbb{R}^{2}}
\sum_{1\leq l,m,l',m'\leq 2} &\frac{1+(-1)^{l}k}{2}f_{1,l}\left(x+\frac{p}{2}\right)i^{m}
g_{1,m}\left(x-\frac{p}{2}\right)\\
&\quad\quad\times i^{l'}g_{2,l'}\left(x-\frac{p}{2}\right)f_{2,m'}
\left(x+\frac{p}{2}\right)\frac{1+(-1)^{m'}k}{2}dxdp
\end{align*}
\[=\int_{\mathbb{R}^{2}}\int_{\mathbb{R}^{2}}f_1\left(x+\frac{p}{2}\right) \overline{g_1\left(x-\frac{p}{2}\right)}g_2\left(x-\frac{p}{2}\right)\overline{f_2\left(x+\frac{p}{2}\right)}dxdp. \qquad\qquad\qquad \]
Let $u=x+\frac{p}{2}$ and $v=x-\frac{p}{2}$, then we can conclude that
\[\begin{aligned}
\left\langle W\left(f_{1}, g_{1}\right), W\left(f_{2}, g_{2}\right)\right\rangle=&\int_{\mathbb{R}^{2}}\int_{\mathbb{R}^{2}} f_1(u) \overline{g_1(v)}g_2(v)\overline{f_2(u)}d u d v
=\left\langle f_{1},\, f_{2} \left\langle \bar{g_2},\, \bar{g_1}\right\rangle\right\rangle.
\end{aligned}\]
\end{proof}

\noindent\textit{Zeros of the Wigner transform.} Now we confer some examples of pairs $(f,g)$
for which the Wigner transform $W(f,g)$ never vanish. The zero set of the Wigner transform is
useful in studying the injectivity of a general Berezin transform and the generalized Berezin
quantization problem. In \cite{GJM}, Authors study under which conditions the Euclidean Wigner
transform never vanish. It is shown that when $f$ and $g$ are generalized Gaussian then the
Euclidean Wigner transform is nonzero. Moreover, from the basic example of the one-sided exponential
function $e^{-at}\boldsymbol{1}_{(0,\infty)}(t)$ some pairs are obtained.

Let $f,g\in L(\mathbb{R}^2)$ be such that Euclidean Wigner transform does not vanish. Then consider
\begin{equation*}
W(f, (a+bk)g)(x, \xi)=\int_{\mathbb{R}^{2}} e^{-2\pi i \xi_1 p_1} f\left(x+\frac{p}{2}\right)
\overline{g\left(x-\frac{p}{2}\right)}(a-bk) e^{-2\pi j \xi_2 p_2}dp,
\end{equation*}
where $a,b\in\mathbb{R}.$ By using (\ref{exp520}), we get
\begin{equation*}
W(f, (a+bk)g)(x, \xi)=\int_{\mathbb{R}^{2}} e^{-2\pi i \xi\cdot p}
f\left(x+\frac{p}{2}\right) \overline{g\left(x-\frac{p}{2}\right)}(a-bk)dp.
\end{equation*}
Hence we gain that the pair $(f, (a+bk)g)$ makes Wigner transform into zero free.
It is also clear that Gaussian $\varphi$ defined in (\ref{exp502}) makes so.
Further, we give examples of such pairs, which are generalized from the Gaussian.
For $a,b,c,d\in\mathbb{R}^2$ define
\[\pi^i_{b_1-d_1}(a)f(x)=e^{2\pi i(b_1-d_1)x_1}f(x-a),~\pi^j_{b_2-d_2}(c)f(x)
=e^{-2\pi j(b_2-d_2)x_2}f(x-c).\]
Then the Wigner transform $W(\pi^i_{b_1-d_1}(a)f,\pi^j_{b_2-d_2}(c)g)(x,\xi)$ is equal to
\[\int_{\mathbb R^{2}}e^{2\pi i(b_1-d_1)(x_1+\frac{p_1}{2})-\xi_1p_1}f(x-a+\frac{p}{2})
\overline{g(x-c-\frac{p}{2})}e^{2\pi j(b_2-d_2)(x_2-\frac{p_2}{2})-\xi_2p_2}dp.\]
Replacing $p$ by $p+a-c$ the above integral can be written as
\begin{align*}	e^{2\pi i\{x_1(b_1-d_1)-\xi_1(a_1-c_1)+\frac{1}{2}(a_1d_1-b_1c_1)\}}W(f,g)
&\left(x-\frac{a+c}{2},\xi-\frac{b+d}{2}\right)\\	
&e^{2\pi j\{x_2(b_2-d_2)-\xi_2(a_2-c_2)+\frac{1}{2}(a_2d_2-b_2c_2)\}}.
\end{align*}
By a suitable change of variable we can conclude that $W(f,g)\neq 0$ if only if
$W(\pi^i_{b_1}(a)f,\pi^j_{b_2}(c)g)\neq 0.$

\subsection{The Weyl transform.}
In this subsection, we introduce the Weyl transform and see its connection with the Wigner transform.

Let $\sigma\in \mathcal{S}\left(\mathbb{R}^{4},\mathbb{Q}\right)$ be a symbol. Then define the Weyl transform $W_\sigma$ on $\mathcal{S}\left(\mathbb{R}^{2},\mathbb{Q}\right)$ corresponding to $\sigma$ by
\begin{equation}\label{exp515}
\left\langle W_{\sigma} f, g\right\rangle=\int_{\mathbb{R}^{2}} \int_{\mathbb{R}^{2}} \sigma(x, \xi) W(f,g)(x, \xi) d x d \xi,
\end{equation}
where $f, g \in \mathcal{S}\left(\mathbb{R}^{2},\mathbb{Q}\right).$
It can be seen that $W_{\sigma}: \mathcal{S}\left(\mathbb{R}^{2},\mathbb{Q}\right)
\rightarrow \mathcal{S}\left(\mathbb{R}^{2},\mathbb{Q}\right)$ is continuous.
\begin{theorem}\label{th501}
Let $\sigma\in L^{r}\left(\mathbb{R}^{4},\mathbb{Q}\right)$ for $1 \leq r \leq 2.$ Then the operator $W_\sigma:L^{2}\left(\mathbb{R}^{2},\mathbb{Q}\right) \rightarrow L^{2}\left(\mathbb{R}^{2},\mathbb{Q}\right)$ is bounded and $\|W_\sigma\|\leq \|\sigma\|_r.$
\end{theorem}
\begin{proof}
The proof will directly follow from Proposition \ref{prop502}.
\end{proof}
The definition of Weyl transform (\ref{exp515}) is defined as involving Wigner transform. The following result illuminates the Weyl transform independently, and we see that the Weyl transform is compact for some class of symbols.

\begin{theorem}\label{th504}
Let $\sigma \in\mathcal{S}\left(\mathbb{R}^{4},\mathbb{Q}\right)$ such that $\sigma(x, \xi)=\sigma(x, -\xi)$. Then for $\varphi\in L^2\left(\mathbb{R}^{2},\mathbb{Q}\right),$
\[\left(W_{\sigma} \varphi\right)(v)=\int_{\mathbb{R}^{2}} \int_{\mathbb{R}^{2}} e^{-2\pi i \xi_1 (v_1-u_1)}\sigma(\frac{u+v}{2}, \xi) e^{-2\pi j \xi_2 (v_2-u_2)}\varphi(u) d ud \xi.\]
\end{theorem}
\begin{proof}
Let $\varphi,\,\psi\in L^2\left(\mathbb{R}^{2},\mathbb{Q}\right).$ From (\ref{exp515}), we have
\begin{align*}
\left\langle W_{\sigma} \varphi, \psi\right\rangle&=\int_{\mathbb{R}^{2}} \int_{\mathbb{R}^{2}} \sigma(x, \xi) W(\varphi,\psi)(x, \xi) d x d \xi\\
&=\int_{\mathbb{R}^{2}}\int_{\mathbb{R}^{2}} \int_{\mathbb{R}^{2}}\sigma(x, \xi) e^{-2\pi i \xi_1 p_1} \varphi\left(x+\frac{p}{2}\right) \overline{\psi\left(x-\frac{p}{2}\right)} e^{-2\pi j \xi_2 p_2}
dp\,dx\,d \xi.
\end{align*}
By the change of variables $u=x+\frac{p}{2}$ and $v=x-\frac{p}{2},$ the above integral becomes
\[\int_{\mathbb{R}^{2}}\int_{\mathbb{R}^{2}} \int_{\mathbb{R}^{2}}\sigma(\frac{u+v}{2}, \xi) e^{-2\pi i \xi_1 (v_1-u_1)} \varphi(u) \overline{\psi(v)} e^{-2\pi j \xi_2 (v_2-u_2)}dudvd\xi.\]
Using arguments like (\ref{exp516},~\ref{exp517}) and the fact that $\sigma(x,-\xi)=\sigma(x, \xi)$ allow
{\small \[\left\langle W_{\sigma} \varphi, \psi\right\rangle=\int_{\mathbb{R}^{2}}\left( \int_{\mathbb{R}^{2}} \int_{\mathbb{R}^{2}} e^{-2\pi i \xi_1 (v_1-u_1)}\sigma(\frac{u+v}{2}, \xi) e^{-2\pi j \xi_2 (v_2-u_2)}\varphi(u) d ud \xi\right) \overline{\psi(v)}   d v.\]}
Hence, we can conclude that
\[\left(W_{\sigma} \varphi\right)(v)=\int_{\mathbb{R}^{2}} \int_{\mathbb{R}^{2}} e^{-2\pi i \xi_1 (v_1-u_1)}\sigma(\frac{u+v}{2}, \xi) e^{-2\pi j \xi_2 (v_2-u_2)}\varphi(u) d ud \xi.\]
\end{proof}

\begin{theorem}
Let $\sigma \in L^{1}\left(\mathbb{R}^{4},\mathbb{Q}\right)$ such that $\sigma(x, \xi)=\sigma(x, -\xi).$ Then $W_{\sigma}: L^{2}\left(\mathbb{R}^{2},\mathbb{Q}\right) \rightarrow L^{2}\left(\mathbb{R}^{2},\mathbb{Q}\right)$ is a trace class operator.
\end{theorem}
\begin{proof}
Because of above Theorem \ref{th504}, $W_{\sigma}$ is an integral operator with the kernel
\begin{equation*}
K(u,v)=\int_{\mathbb{R}^{2}} e^{-2\pi i \xi_1 (v_1-u_1)}\sigma(\frac{u+v}{2}, \xi)e^{-2\pi j \xi_2 (v_2-u_2)}d \xi.
\end{equation*}
For $\sigma\in L^{1}\left(\mathbb{R}^4,\mathbb{Q}\right),$ we have
\begin{equation*}
\|W_\sigma\|_{S_1}=\int_{\mathbb{R}^{2}}|K(u,u)|du\leq
\int_{\mathbb{R}^{2}}\int_{\mathbb{R}^{2}}\left|\sigma(u, \xi)\right|dud\xi=\|\sigma\|_1.
\end{equation*}
Hence we conclude that $W_\sigma$ is a trace class operator.
\end{proof}
\begin{theorem}
Let $\sigma \in L^{r}\left(\mathbb{R}^{4},\mathbb{Q}\right), 1 \leq r \leq 2,$ such that $\sigma(x, \xi)=\sigma(x, -\xi).$ Then $W_{\sigma}: L^{2}\left(\mathbb{R}^{2},\mathbb{Q}\right) \rightarrow L^{2}\left(\mathbb{R}^{2},\mathbb{Q}\right)$ is a compact operator.
\end{theorem}
\begin{proof}
In view of Theorem \ref{th504}, we obtain that $W_{\sigma}$ is an integral operator with the kernel
\begin{equation*}
K(u,v)=\int_{\mathbb{R}^{2}} e^{-2\pi i \xi_1 (v_1-u_1)}\sigma(\frac{u+v}{2}, \xi)e^{-2\pi j \xi_2 (v_2-u_2)}d \xi.
\end{equation*}
Suppose $\sigma\in L^{2}\left(\mathbb{R}^4,\mathbb{Q}\right),$ then the Hilbert-Schmidt norm of $W_\sigma$ will be
\begin{equation}\label{exp508}
\begin{aligned}
\|W_\sigma\|_{S_2}^2=&\int_{\mathbb{R}^{2}}\int_{\mathbb{R}^{2}}|K(u,v)|^2dudv=\int_{\mathbb{R}^{2}}\int_{\mathbb{R}^{2}}\left|K(x-\frac{p}{2},x+\frac{p}{2})\right|^2dxdp\\
=&\int_{\mathbb{R}^{2}}\int_{\mathbb{R}^{2}}\left|\int_{\mathbb{R}^{2}} e^{-2\pi i \xi_1 p_1}\sigma(x, \xi)e^{-2\pi j \xi_2 p_2}d \xi\right|^2dxdp\\
=&\int_{\mathbb{R}^{2}}\left(\int_{\mathbb{R}^{2}}|\left(\mathcal{F}_2\sigma(x,\cdot)\right)(p)|^2dp\right)dx,
\end{aligned}
\end{equation}
where $\mathcal{F}_2\sigma$ denotes the QFT in second variable. Now by the Plancherel theorem we can conclude that $\|W_\sigma\|_{S_2}^2=\|\sigma\|_2^2.$

Let $\sigma \in L^{r}\left(\mathbb{R}^{4},\mathbb{Q}\right), 1 \leq r \leq 2 .$ Then we can choose
a sequence $\left\{\sigma_{k}\right\}_{k=1}^{\infty}$ of functions in $\mathcal{S}\left(\mathbb{R}^{4},\mathbb{Q}\right)$
such that $\sigma_{k} \rightarrow \sigma$ in $L^{r}\left(\mathbb{R}^{4},\mathbb{Q}\right)$ as
$k \rightarrow \infty$. Therefore, for each $k\in\mathbb{N}$,  $W_{\sigma_{k}}$ is a Hilbert-Schmidt
operator, and hence it is compact. By Theorem (\ref{th501}), $W_{\sigma}$ is the limit of the sequence
$\left\{W_{\sigma_{k}}\right\}_{k=1}^{\infty}$ in the space of bounded linear operators on
$L^{2}\left(\mathbb{R}^{2},\mathbb{Q}\right).$ Thus, $W_{\sigma}$ is compact.
\end{proof}

Next results will give necessary and sufficient condition for boundedness of the Weyl transform for $r>2.$
Now onwards, in this section, $r$ and $r'$ are conjugate indices i.e., $\frac{1}{r}+\frac{1}{r'}=1.$
\begin{proposition}\label{prop501}
The following two statements are equivalent:	
\begin{enumerate}
\item The Weyl transform $W_{\sigma}$ is a bounded linear operator on
$L^{2}\left(\mathbb{R}^{2},\mathbb{Q}\right)$ for all
$\sigma \in L^{r}\left(\mathbb{R}^{4},\mathbb{Q}\right),$ where $ 2<r<\infty.$
\item There exists a positive constant $C$ such that
$\|W(f, g)\|_{r'}\leq C\|f\|_2\|g\|_2$ for all $f,g\in L^2\left(\mathbb{R}^{2},\mathbb{Q}\right).$
\end{enumerate}	
\end{proposition}
Since the proof of the above Proposition will continue with the same line as in the Euclidean case
(using the uniform boundedness principle), we will skip it here.
\begin{proposition}
Let $f\in L^2\left(\mathbb{R}^{2},\mathbb{Q}\right)$ be a compactly supported function such that $\int_{\mathbb{R}^2} f(x) dx \neq 0.$
If $W_{\sigma}$ is a bounded operator on $L^2\left(\mathbb{R}^{2},\mathbb{Q}\right)$
for all $\sigma \in L^{r}\left(\mathbb{R}^{4},\mathbb{Q}\right),$ where $2<r<\infty,$ then $\|\mathcal{F}(f)\|_{r'}<\infty.$
\end{proposition}
\begin{proof}
Let $f\in L^2\left(\mathbb{R}^{2},\mathbb{Q}\right)$ be supported on the unit disk
$\left\{x \in \mathbb{R}^{2}:|x| \leq 1\right\}.$ Consider the functions $\tilde{f}^i$ defined in (\ref{exp506}) and $\bar{\tilde{f}}(x)=\overline{f(-x)}.$ Then $\tilde{f}^i,\,\bar{\tilde{f}}\in L^2\left(\mathbb{R}^{2},\mathbb{Q}\right)$ and both are supported on the unit disk. Further, $W(\tilde{f}^i,\bar{\tilde{f}})(x, \xi) \neq 0,$ if
$|x| \leq \frac{1}{2}\left|\left(x+\frac{p}{2}\right)+\left(x-\frac{p}{2}\right)\right| \leq 1$
and $W(\tilde{f}^i,\bar{\tilde{f}})(x, \xi)=0$ for all $\xi\in\mathbb{R}^2,$ if $|x|>1.$ By
Proposition \ref{prop501}, we get
\[\int_{\mathbb{R}^{2}} \int_{\mathbb{R}^{2}}|W(\tilde{f}^i,\bar{\tilde{f}})(x, \xi)|^{r^{\prime}} d x d \xi<\infty.\]
Therefore, by the Minkowshki's integral inequality
\begin{align*}
&\left(\int_{\mathbb{R}^{2}}\left|\int_{|x| \leq 1}e^{4\pi i\xi_1x_1} W(\tilde{f}^i,\bar{\tilde{f}})(x, \xi)
e^{4\pi j\xi_2x_2} d x\right|^{r^{\prime}} d \xi\right)^{\frac{1}{r^{\prime}}} \\
&\leq \int_{|x| \leq 1}\left(\int_{\mathbb{R}^{2}}|W(\tilde{f}^i,\bar{\tilde{f}})(x, \xi)|^{r^{\prime}}
d \xi\right)^{\frac{1}{r^{\prime}}} d x \\
&\leq\left(\int_{|x| \leq 1} d x\right)^{\frac{1}{r}}\left(\int_{\mathbb{R}^{2}}
\int_{\mathbb{R}^{2}}|W(\tilde{f}^i,\bar{\tilde{f}})(x, \xi)|^{r^{\prime}} d x d \xi\right)^{\frac{1}{r^{\prime}}}<\infty.
\end{align*}

Thus,
\begin{equation}\label{exp518}
\int_{|x| \leq 1}e^{4\pi i\xi_1x_1} W(\tilde{f}^i,\bar{\tilde{f}})(x, \xi) e^{4\pi j\xi_2x_2}
dx\in L^{r^{\prime}}\left(\mathbb{R}^{2},\mathbb{Q}\right).
\end{equation}
But
\begin{align*}
&\int_{|x| \leq 1}e^{4\pi i\xi_1x_1} W(\tilde{f}^i,\bar{\tilde{f}})(x, \xi) e^{4\pi j\xi_2x_2} d x \\
=&\int_{\mathbb{R}^{2}} \int_{\mathbb{R}^{2}} e^{4\pi i \xi_1\left(x_1-\frac{p_1}{2}\right)} \tilde{f}^i\left(x+\frac{p}{2}\right) \tilde{f}\left(x-\frac{p}{2}\right)e^{4\pi j \xi_2\left(x_2-\frac{p_2}{2}\right)} d x d p\\
=&\int_{\mathbb{R}^{2}} \int_{\mathbb{R}^{2}} e^{4\pi i \xi_1v_1} \tilde{f}^i(u) \tilde{f}(v)e^{4\pi j \xi_2v_2} d u d v,
\end{align*}
where the last equality follows by the change of variables $u=x+\frac{p}{2}\mbox{ and }v=x-\frac{p}{2}.$
A similar calculation as in (\ref{exp517}) leads to
\begin{align*}
\int_{|x| \leq 1}e^{4\pi i\xi_1x_1} W(\tilde{f}^i,\bar{\tilde{f}})(x, \xi) e^{4\pi j\xi_2x_2} d x
&=\int_{\mathbb{R}^{2}} f(u)\int_{\mathbb{R}^{2}} e^{-4\pi i \xi_1v_1} f(v)e^{-4\pi j \xi_2v_2} dv du \\
&=\mathcal{F}(f)(0) \mathcal{F}(f)(2 \xi).
\end{align*}
From (\ref{exp518}), it follows that $\|\mathcal{F}(f)\|_{r'}<\infty.$
\end{proof}

Now, we will give an example of the above described function.

\begin{example}
Let's first recall the example of such function for the Euclidean Fourier transform, which was given
by Simon \cite{Si}. Consider the cube $$Q=\left\{x \in \mathbb{R}^{n}:-a \leq x_{j} \leq a,~ j=1,2, \ldots, n\right\}$$
lying inside $\left\{x \in \mathbb{R}^{n}:|x| \leq 1\right\}.$ For $\alpha \in\left(0, \frac{1}{2}\right)$
consider the function $g_\alpha$ on $\mathbb{R}^{n}$ defined by
\begin{equation}\label{exp512}
g_\alpha(x)=\left\{\begin{array}{ll}
	\prod_{j=1}^{n}\left|x_{j}\right|^{-\alpha}, & x \in Q\setminus \{0\} \\
	0, & \text { otherwise. }
\end{array}\right.
\end{equation}
Here $g_\alpha$ is square integrable, compactly supported real valued function on $\mathbb{R}^n$ with
$\int_{\mathbb{R}^n} g_\alpha(x) dx \neq 0.$ For each $1<r'<2,$ there exists $\alpha\in\left(0, \frac{1}{2}\right)$ such that $\int_{\mathbb{R}^n}|\hat{g}_\alpha(\xi)|^{r'} d\xi=\infty,$
where $\hat{g}_\alpha$ is the Euclidean Fourier transform of $g_\alpha.$

\smallskip

Now, we will modify  the above example
for quaternion Fourier transform on $\mathbb{R}^2.$ Let $Q$ be the cube on $\mathbb{R}^2$ as defined
above and $\alpha \in\left(0, \frac{1}{2}\right).$ Then define
\begin{equation}
f_\alpha(x)=g_\alpha(x)\frac{1-k}{2},
\end{equation}
where $g_\alpha(x)$ is a function on $\mathbb{R}^2$ define in (\ref{exp512}). From (\ref{exp513}) we get
\[\mathcal{F}(f_\alpha)(\xi)=\int_{\mathbb{R}^{2}} e^{-2 \pi i\left(x_{1} \xi_{1} + x_{2} \xi_{2}\right)}
f_\alpha(x_1,x_2) d x_{1} d x_{2} =\hat{g}_\alpha(\xi)\frac{1-k}{2}.\]
Hence from the Euclidean setup it follows that for each $1<r'<2,$ there exists $\alpha\in\left(0, \frac{1}{2}\right)$ such that $\int_{\mathbb{R}^2}|\mathcal{F}(f_\alpha) (\xi)|^{r'} d\xi=\infty.$
\end{example}

\section{Uniqueness results}\label{sec4}
\subsection{Benedicks-Amrein-Berthier type theorem for the Weyl transform.}
A fruitful method of working on certain problems in the Heisenberg group is that, instead of considering
the group Fourier transform, it is enough to consider the following Weyl transform.
The Weyl transform of $h\in L^1(\mathbb C^n)$ is defined by
\begin{align}\label{exp701}
\mathcal{W}(h)=\int_{\mathbb C^n} h(z)\pi(z) dz,
\end{align}
where for $z=x+iy\in \mathbb C^n,$ $\pi(z)$ is an operator on $L^2(\mathbb R^n)$ defined by
\begin{align}
\pi(z)\psi(\xi)=e^{2\pi i(x\cdot\xi+\frac{1}{2}x\cdot y)}\psi(\xi+y).
\end{align}
Further for $\psi_1,\psi_2\in L^2(\mathbb R^n),$ the following square integrability relation holds
\begin{equation}\label{exp705}
\int_{\mathbb C^n}|\langle \pi(z)\psi_1,\psi_2 \rangle|^2 dz=\|\psi_1\|^2 \|\psi_2\|^2.
\end{equation}
In addition, the inversion formula for the Weyl transform holds (see \cite{T1}),
\begin{equation}\label{exp703}
h(z)=\text{tr}\left(\mathcal{W}(h)\pi(-x,-y)\right)
\end{equation}
and satisfies the Plancherel formula
\begin{equation}\label{exp711}
\|\mathcal{W}(h)\|_{HS}=\|h\|_2.
\end{equation}
In \cite{NR,V,GS}, through different approaches, it is proved that a non-zero function $h\in L^1(\mathbb C^n)$
supported in a finite measure set can not have finite rank Weyl transform $\mathcal{W}(h).$
In this section, we prove an analogue result for the Weyl transform defined in terms of QFT.

\smallskip

Consider the Weyl transform defined in (\ref{exp515}) and assume that $\sigma$ is even. From
(\ref{exp522}) and (\ref{exp509}) we have
{\footnotesize \[\langle W_{\sigma}\varphi,\psi\rangle =\int_{\mathbb{R}^{2}}\int_{\mathbb R^2}\sigma(x,\xi)
\int_{\mathbb{R}^{2}} \int_{\mathbb R^2} e^{-2\pi i(x_1q_+\xi_1p_1)}V(\varphi, \psi)(q, p)
e^{-2\pi j(x_2q_2+\xi_2p_2)} d q d p \, dx d\xi.\]}
By proceeding through similar techniques as in (\ref{exp516}, \ref{exp517}) and then using
$\sigma(-x,-\xi)=\sigma(x,\xi)$ we have
{\footnotesize \begin{align*}
\langle W_{\sigma}\varphi,\psi\rangle &=\int_{\mathbb{R}^{2}}\int_{\mathbb R^2}\left(\int_{\mathbb{R}^{2}} \int_{\mathbb R^2}
e^{-2\pi i(x_1q_1+\xi_1p_1)}\sigma(x, \xi)e^{-2\pi j(x_2q_2+\xi_2p_2)} dx d\xi\right)
V(\varphi, \psi)(q, p) d q d p \\
&=\int_{\mathbb{R}^{2}}\int_{\mathbb R^2}(\mathcal{F}\sigma)(q,p)\int_{\mathbb{R}^{2}}
e^{2 \pi i\left(q_{1} x_{1}+\frac{1}{2} q_{1} p_{1}\right)}{\varphi(x+p)} \overline{\psi(x)}
e^{2 \pi j\left(q_2 x_{2}+\frac{1}{2} q_{2} p_{2}\right)} d x \, d q d p,
\end{align*}}
where the last equality holds from (\ref{exp514}). Again using an argument like (\ref{exp516}) and the
fact that $\sigma(-x,-\xi)=\sigma(x, \xi)$ implies $(\mathcal{F}\sigma)(-q,-p)=(\mathcal{F}\sigma)(q,p),$ leads to
{\footnotesize \begin{equation}\label{exp710}
\langle W_{\sigma}\varphi,\psi\rangle=\int_{\mathbb{R}^{2}}\left(\int_{\mathbb R^2}\int_{\mathbb{R}^{2}}
(\mathcal{F}\sigma)(q,p)e^{2 \pi i\left(q_{1} x_{1}+\frac{1}{2} q_{1} p_{1}\right)}{\varphi(x+p)}
e^{2 \pi j\left(q_2 x_{2}+\frac{1}{2} q_{2} p_{2}\right)} d q d p \right) \overline{\psi(x)} d x.
\end{equation}}
Thus, in view of (\ref{exp701}) and (\ref{exp710}), we can define the following Weyl type transform.
To specify, for each $x=(x_1,x_2),~ y=(y_1,y_2)\in\mathbb R^2,$ consider the operator
$\rho(x,y):L^2(\mathbb R^2,\mathbb Q)\rightarrow L^2(\mathbb R^2,\mathbb Q)$
defined by
\begin{align}
\rho(x,y)\varphi(\xi)=e^{2\pi i(x_1\xi_1+\frac{1}{2}x_1y_1)}\varphi(\xi+y)e^{2\pi j(x_2\xi_2+\frac{1}{2}x_2y_2)},
\end{align}
where $\varphi\in L^2(\mathbb R^2,\mathbb Q).$

Let $x,y,u,v\in \mathbb R^2.$ Then the following relation holds
\begin{align}\label{exp709}
\rho(u,v)\rho(x,y)\varphi(\xi)=e^{\pi i(x_1v_1-u_1y_1)}\rho(u+x,v+y)\varphi(\xi)e^{\pi j(x_2v_2-u_2y_2)},
\end{align}
where $\varphi\in L^2(\mathbb R^2,\mathbb Q)$ and $\xi\in \mathbb R^2.$

For $g\in L^1(\mathbb R^4,\mathbb Q),$ define the Weyl transform of $g$ by
\begin{align*}
W(g)=\int_{\mathbb R^2}\int_{\mathbb R^2}g(x,y)\rho(x,y)dx dy.
\end{align*}

Consider the subspaces $L_\pm^2(\mathbb R^2,\mathbb Q)=\{\varphi_\pm:\varphi\in L^2(\mathbb R^2,\mathbb Q)\}$
and let $P_\pm$ be the projections of $L^2(\mathbb R^2,\mathbb Q)$ onto $L_\pm^2(\mathbb R^2,\mathbb Q),$ respectively.
Since $(1-k)(a+jb)=(a+ib)(1-k)$ for $a,b\in\mathbb R,$ we have
\begin{equation}\label{exp702}
\rho(u,v)\rho(x,y)\varphi_{-}=\pi(u,v)\pi(x,y)\varphi_{-}.
\end{equation}

Next, we shall see that $g$ can be retrieved from the Weyl transform by the following inversion formula.
\begin{theorem}{\em (Inversion formula)}\label{th701}
Let $g\in L^1(\mathbb R^4,\mathbb Q).$ Then
\[g(x,y)=\text{tr}\left(W(g)\rho(-x,-y)P_{-}\right)\text{ a.e.}\]
\end{theorem}

\begin{proof}
Consider an orthonormal basis $\{\tilde{e}_l:l\in\mathsf{\Lambda}\}$ of $L^2(\mathbb R^2).$ Then
$\mathsf{B}=\{e_l=\tilde{e}_l(\frac{1-k}{\sqrt{2}}):l\in\mathsf{\Lambda}\}$ is an orthonormal basis
of $L_{-}^2(\mathbb R^2,\mathbb Q).$ Extend $\mathsf{B}$ to an orthonormal basis $\{e_l:l\in \mathbf{\Lambda}\}$
of $L^2(\mathbb R^2,\mathbb Q),$ where $\mathsf{\Lambda}\subset \mathbf{\Lambda}.$ Now,
\begin{align*}
\text{tr}\left(W(g)\rho(-x,-y)P_{-}\right)
&=\sum_{l\in \mathbf{\Lambda}}\left\langle W(g)\rho(-x,-y)P_{-} e_l,e_l \right\rangle
=\sum_{l\in \mathsf{\Lambda}}\left\langle W(g)\rho(-x,-y) e_l,e_l \right\rangle \\
&=\sum_{l\in \mathsf{\Lambda}}\left(\left\langle W(h_1)\rho(-x,-y) e_l,e_l \right\rangle
+j\left\langle W(h_2)\rho(-x,-y) e_l,e_l \right\rangle\right),
\end{align*}
where $g=h_1+jh_2$ and $h_1,h_2$ are complex valued functions. In view of (\ref{exp701})
and (\ref{exp702}), we have
\begin{align*}
\text{tr}\left(W(g)\rho(-x,-y)P_{-}\right)
&=\sum_{l\in \mathsf{\Lambda}}\left(\left\langle \mathcal{W}(h_1)\pi(-x,-y) e_l,e_l \right\rangle
+j\left\langle \mathcal{W}(h_2)\pi(-x,-y) e_l,e_l \right\rangle\right) \\
&=\sum_{l\in \mathsf{\Lambda}}\left(\left\langle \mathcal{W}(h_1)\pi(-x,-y) \tilde{e}_l,\tilde{e}_l \right\rangle
+j\left\langle \mathcal{W}(h_2)\pi(-x,-y) \tilde{e}_l,\tilde{e}_l \right\rangle\right),
\end{align*}
where the last equality holds as $\left(\frac{1-k}{\sqrt{2}}\right)\left(\frac{1+k}{\sqrt{2}}\right)=1.$ Hence,
\begin{align*}
\text{tr}\left(W(g)\rho(-x,-y)P_{-}\right)=\text{tr}\left(\mathcal{W}(h_1)\pi(-x,-y)\right)
+j\text{ tr}\left(\mathcal{W}(h_2)\pi(-x,-y)\right).
\end{align*}
Thus, by (\ref{exp703}) we get
\begin{align*}
\text{tr}\left(W(g)\rho(-x,-y)P_{-}\right)=h_1(x,y)+jh_2(x,y)=g(x,y).
\end{align*}
\end{proof}
Let $g\in L^2(\mathbb R^4,\mathbb Q)$ and write $g=h_1+jh_2$ for complex valued functions $h_1$ and $h_2.$
If we go through the proof of Theorem \ref{th701}, then we can have
\begin{equation}\label{exp712}
\|W(g)P_{-}\|_{HS}^2=\|\mathcal{W}(h_1)\|_{HS}^2+\|\mathcal{W}(h_2)\|_{HS}^2=\|g\|_2^2,
\end{equation}
where the last equality follows from the Plancherel formula (\ref{exp711}).

Consider the set $A$ in $\mathbb{R}^4$ of the form
\begin{equation}\label{exp714}
A=A^{(1,3)}\times A^{(2,4)},
\end{equation}
where
$A^{(1,3)}=\{(x_1,y_1)\in \mathbb R^2:(x_1,x_2,y_1,y_2)\in A\}$ and
$A^{(2,4)}=\{(x_2,y_2)\in \mathbb R^2:(x_1,x_2,y_1,y_2)\in A\}$ are the
projection of $A$ in $X_1 Y_1$-plane and $X_2 Y_2$-plane, respectively.
Then the following is the main result of the section.
\begin{theorem}\label{th702}
Let $g\in L^1(\mathbb R^4,\mathbb Q)$ and $\{(x,y)\in \mathbb{R}^4:g(x,y)\neq 0\} \subseteq A,$
where $A=A^{(1,3)}\times A^{(2,4)}$ is defined as in (\ref{exp714}). Suppose $A^{(1,3)}$ and $A^{(2,4)}$ have finite $2$-dimensional Lebesgue measure.
If $W(g)$ has finite rank, then $g=0.$
\end{theorem}
Let $g\in L^2(\mathbb R^4,\mathbb Q)$ and $W(g)$ be a finite rank operator. Then there exists an orthonormal
basis $\{\varphi_1,\varphi_2,\ldots\}$ of $L^2(\mathbb R^2,\mathbb Q)$ such that $\mathcal{R}(W(g))=S,$
where $S=\text{span}\{\varphi_1,\ldots,\varphi_N \}$ and $\mathcal R$ stands for the range. Define an orthogonal
projection $P_S$ of $L^2(\mathbb R^2,\mathbb Q)$ onto $S.$ Let $A$ be a measurable subset of $\mathbb{R}^4.$
Define a pair of orthogonal projections $E_A$ and $F_S$ of $L^2(\mathbb R^4,\mathbb Q)$ by
\begin{equation}\label{exp704}
E_A g=\chi_A g \qquad \text{ and } \qquad W(F_S g)=P_S W(g),
\end{equation}
where $\chi_A$ denotes the characteristic function of $A.$ Then
$\mathcal{R}(E_A)=\{g\in L^2(\mathbb R^4,\mathbb Q): g= \chi_A g \}$ and
$\mathcal{R}(F_S)=\{g\in L^2(\mathbb R^4,\mathbb Q): \mathcal{R}(W(g))\subseteq S\}.$

\smallskip

Now we compute the Hilbert-Schmidt norm of $E_AF_S,$ when $m(A)$ is finite.
\begin{lemma}\label{lemma701}
The operator $E_AF_S$ is an integral operator with kernel
$K(x,y,u,v)=\chi_A(x,y)\text{tr}\left(P_S \rho(u,v)\rho(-x,-y)P_{-}\right),$
where $x,y,u,v\in\mathbb R^2.$
\end{lemma}

\begin{proof}
For $g\in L^2(\mathbb R^4,\mathbb Q)$, we have $W(F_S g)=P_S W(g)$. Then by Theorem \ref{th701}, inversion formula
for the Weyl transform, we have
\begin{align*}
(F_Sg)(x,y)&=\text{tr}\left(W(F_Sg)\rho(-x,-y)P_{-}\right)
=\text{tr}\left(P_SW(g)\rho(-x,-y)P_{-}\right) \\
&=\int_{\mathbb{R}^4} g(u,v) \text{tr}\left(P_S \rho(u,v)\rho(-x,-y)P_{-}\right)du dv.
\end{align*}
Hence, we can write
\begin{align*}
(E_AF_Sg)(x,y)&=\chi_A(x,y)(F_Sg)(x,y) \\
&=\chi_A(x,y)\int_{\mathbb{R}^4} g(u,v) \text{tr}\left(P_S \rho(u,v)\rho(-x,-y)P_{-}\right)du dv \\
&=\int_{\mathbb{R}^4} g(u,v) K(x,y,u,v)dudv,
\end{align*}
where $K(x,y,u,v)=\chi_A(x,y)\text{tr}\left(P_S \rho(u,v)\rho(-x,-y)P_{-}\right).$
\end{proof}

\begin{lemma}\label{lemma702}
$E_AF_S$ is Hilbert-Schmidt satisfying $\|E_AF_S\|_{HS}^2\leq m(A)N^2,$
where $N=\dim(S).$
\end{lemma}

\begin{proof}
It is proved in Lemma \ref{lemma701} that $E_AF_S$ is an integral operator with kernel $K(x,y,u,v).$ Therefore,
\begin{align}\label{exp706}
\|E_AF_S\|_{HS}^2&= \int_{\mathbb{R}^4}\int_{\mathbb{R}^4}|K(x,y,u,v)|^2 du dv dx dy \nonumber \\
&=\int_{\mathbb{R}^4}|\chi_A(x,y)|^2 \left(\int_{\mathbb{R}^4}
|\text{tr}\left(P_S \rho(u,v)\rho(-x,-y)P_{-}\right)|^2 du dv\right)dx dy \nonumber \\
&=\int_{\mathbb{R}^4}\chi_A(x,y) \left(\int_{\mathbb{R}^4}\Big|
\sum\limits_{j=1}^N \langle \rho(u,v)\rho(-x,-y)P_{-} \varphi_j,\varphi_j\rangle \Big|^2 du dv\right)dx dy.
\end{align}
In view of (\ref{exp702}), the above inner integral reduces to
\begin{align*}
\int_{\mathbb{R}^4}\Big|\sum\limits_{j=1}^N \langle \pi(u,v)\pi(-x,-y)P_{-} \varphi_j,\varphi_j\rangle \Big|^2 du dv
=\int_{\mathbb{R}^4}\Big|\sum\limits_{j=1}^N \langle \pi(u,v)\psi_j^{x,y},\varphi_j\rangle \Big|^2 du dv,
\end{align*}
where $\psi_j^{x,y}=\pi(-x,-y)P_{-} \varphi_j$ and $\|\psi_j^{x,y}\|\leq \|\varphi_j\|=1$
for all $x,y\in\mathbb R^2.$ Then by using (\ref{exp705}), we get
\begin{align*}
\int_{\mathbb{R}^4}\Big|\sum\limits_{j=1}^N \langle \pi(u,v)\psi_j^{x,y},\varphi_j\rangle \Big|^2 du dv
\leq N\sum\limits_{j=1}^N \int_{\mathbb{R}^4}\Big|\langle \pi(u,v)\psi_j^{x,y},\varphi_j\rangle \Big|^2 du dv
\leq N^2.
\end{align*}
Hence from (\ref{exp706}), we have $\|E_AF_S\|_{HS}^2\leq m(A)N^2.$
\end{proof}

We need the following result from \cite{AB} that describes an interesting
property of measurable sets having finite measure. Denote $u B=\{x\in \mathbb{R}^n:x-u\in B\}.$

\begin{lemma}{\em \cite{AB}}\label{lemma703}
Let $B$ be a measurable set in $\mathbb{R}^n$ with $0<m(B)<\infty.$ If $B_0$
is a measurable subset of $B$ with $m(B_0)>0,$ then for $\epsilon>0,$ there exists
$u\in \mathbb{R}^n$ such that \[m(B)<m(B\cup u B_0)<m(B)+\epsilon.\]
\end{lemma}
Let $E$ and $F$ are orthogonal projections on a Hilbert space $\mathcal H.$
Denote $E\cap F$ be the orthogonal projection of $\mathcal H$ onto
$\mathcal{R}(E)\cap \mathcal{R}(F).$ Then, we have the relation
\begin{align}\label{exp707}
\Vert E\cap F \Vert_{HS}^2=\dim\mathcal{R}( E\cap F)\leq \Vert EF \Vert_{HS}^2.
\end{align}

\begin{proposition}\label{prop701}
Let $A=A^{(1,3)}\times A^{(2,4)}$ be defined as in (\ref{exp714}).
Suppose $A^{(1,3)}$ and $A^{(2,4)}$ have finite $2$-dimensional Lebesgue measure.
Then the projection $E_A\cap F_S=0.$
\end{proposition}

\begin{proof}
Assume towards contrary that there exists a non-zero function $g_0\in\mathcal{R}( E_A\cap F_S)$.
Then $\mathcal{R}(W(g_0))\subseteq S.$ Let
\begin{align}\label{exp708}
A_0^{(1,3)}=\{(x_1,y_1)\in A^{(1,3)}: & \, \exists \, I_{(x_1,y_1)}\subseteq A^{(2,4)} \text{ with } m(I_{(x_1,y_1)})>0 \\
&\text{ and } g_0(x_1,x_2,y_1,y_2)\neq 0 \, \forall \, (x_2,y_2)\in I_{(x_1,y_1)}\} \nonumber.
\end{align}
Clearly, $0<m(A_0^{(1,3)})<\infty$. Choose $s\in \mathbb{N}$ such that
$s>2\,m(A_0^{(1,3)})m(A^{(2,4)})N^2,$
where $N=\dim(S).$ Now, we construct an increasing sequence of measurable sets $\{A_l^{(1,3)}:l=1,\ldots,s\}$
containing $A_0^{(1,3)}.$ By Lemma \ref{lemma703}, for $\epsilon=\frac{1}{2m(A^{(2,4)})N^2}$, $B_0=A_0^{(1,3)}$
and $B=A_{l-1}^{(1,3)},$ there exists $u_l=(u_l^{(1)},u_l^{(2)})\in \mathbb{R}^2$ such that
\[m(A_{l-1}^{(1,3)})<m(A_{l-1}^{(1,3)}\cup u_lA_0^{(1,3)})<m(A_{l-1}^{(1,3)})+\frac{1}{2m(A^{(2,4)})N^2}.\]
Write $A_l^{(1,3)}=A_{l-1}^{(1,3)}\cup u_lA_0^{(1,3)}$ and $A_s=A_s^{(1,3)}\times A^{(2,4)}.$
In view of (\ref{exp707}) and Lemma \ref{lemma702}, we obtain
\begin{align}\label{exp715}
\dim\mathcal{R}( E_{A_s}\cap F_S) &\leq m(A_s)N^2=m(A_s^{(1,3)})m(A^{(2,4)})N^2 \nonumber \\
&\leq \left(m(A_0^{(1,3)})+\frac{s}{2m(A^{(2,4)})N^2}\right)m(A^{(2,4)})N^2<s.
\end{align}
Next, we shall produce $s+1$ linearly independent functions in $\mathcal{R}( E_{A_s}\cap F_S).$
Let $$g_l(x,y)=g_0(x_1-u_l^{(1)},x_2,y_1-u_l^{(2)},y_2)e^{\pi i(y_1u_l^{(1)}-x_1u_l^{(2)})}.$$
We shall show that $g_l\in \mathcal{R}(F_S)$ for each $l=1,\ldots,s.$ To do so, let
$\varphi \in L^2(\mathbb{R}^2,\mathbb{Q})$ and $j>N.$ Then
\begin{align*}
\langle W(g_l)\varphi,\varphi_j\rangle &=\int_{\mathbb{R}^4}
g_l(x_1,x_2,y_1,y_2) \langle \rho(x_1,x_2,y_1,y_2)\varphi,\varphi_j\rangle dx_1 dx_2 dy_1 dy_2 \\
&=\int_{\mathbb{R}^4} g_0(x_1-u_l^{(1)},x_2,y_1-u_l^{(2)},y_2)e^{\pi i(y_1u_l^{(1)}-x_1u_l^{(2)})}
\langle \rho(x,y)\varphi,\varphi_j\rangle dx dy \\
&=\int_{\mathbb{R}^4} g_0(x,y) e^{\pi i(y_1u_l^{(1)}-x_1u_l^{(2)})}
\langle \rho(x_1+u_l^{(1)},x_2,y_1+u_l^{(2)},y_2)\varphi,\varphi_j\rangle dx dy.
\end{align*}
From (\ref{exp709}) we have $$\rho(x_1,x_2,y_1,y_2)\rho(u_l^{(1)},0,u_l^{(2)},0)
=e^{\pi i(y_1u_l^{(1)}-x_1u_l^{(1)})}\rho(x_1+u_l^{(1)},x_2,y_1+u_l^{(2)},y_2).$$
Thus,
\begin{align*}
\langle W(g_l)\varphi,\varphi_j\rangle
&=\int_{\mathbb{R}^4} g_0(x,y)\langle \rho(x_1,x_2,y_1,y_2)\rho(u_l^{(1)},0,u_l^{(2)},0)\varphi,\varphi_j\rangle dx dy \\
&=\int_{\mathbb{R}^4} g_0(x,y)\langle \rho(x,y)\psi,\varphi_j\rangle dx dy
=\langle W(g_0)\psi,\varphi_j\rangle=0.
\end{align*}
Hence $\mathcal{R}(W(g_l))\subseteq S.$
Since, $A_m=(A_0^{(1,3)}\cup u_1A_0^{(1,3)}\cup\cdots\cup u_m A_0^{(1,3)})\times A^{(2,4)}$ and
$g_l=0$ on $(u_lA_0^{(1,3)}\times A^{(2,4)})^c$, we have $E_{A_m}g_l=g_l$ for $l=1,\ldots,m$.
Furthermore, $E_{A_m \setminus A_{m-1}}g_l=0$ for $l=1,\ldots,m-1$ and in view of (\ref{exp708}),
$E_{A_m \setminus A_{m-1}}g_m\neq 0$. Therefore $g_m$ cannot be written as
a linear combination of $g_0,\ldots,g_{m-1}$. Thus, $g_0,\ldots,g_s \in \mathcal{R}( E_{A_s}\cap F_S),$
are linearly independent functions, which contradicts \ref{exp715}.
\end{proof}

\begin{proof}[Proof of Theorem \ref{th702}]
For $g\in L^2(\mathbb R^4,\mathbb Q),$
the result follows from Proposition \ref{prop701}. Further, if $g\in L^1(\mathbb R^4,\mathbb Q),$
then (\ref{exp712}) and the finite rank assumption on $W(g)$ give $g\in L^2(\mathbb R^4,\mathbb Q).$
This completes the proof.
\end{proof}

\smallskip

\noindent\textit{Beurling's theorem.} The version of Beurling's theorem for the Fourier-Weyl transform on
step two nilpotent Lie groups proved in \cite{PT}, can be generalized for the quaternion Fourier-Weyl transform.

For $\xi=\left(\xi^{\prime}, \xi^{\prime \prime}\right),$ where $\xi^{\prime}, \xi^{\prime \prime} \in \mathbb{R}^{2},$
define the quaternion Fourier-Weyl transform of $f\in L^1(\mathbb R^4,\mathbb Q)$ by
\begin{align*}
\mathcal{\tilde{F}}(W(f))(\xi)=\int_{\mathbb R^2}\int_{\mathbb R^2}e^{2\pi i(x_1\xi_1^{\prime \prime}-y_1\xi_1^{\prime})}
f(x,y)\rho(x,y)e^{2\pi j(x_1\xi_2^{\prime \prime}-y_1\xi_2^{\prime})}dx dy.
\end{align*}
Then we have the following version of Beurling's theorem.
\begin{theorem}
Let $f\in L^1(\mathbb R^4,\mathbb Q)$ be such that
\begin{equation}\label{exp713}
\int_{\mathbb{R}^{4}} \int_{\mathbb{R}^{4}}\left|f(x, y)\right|
\left|\langle\mathcal{\tilde{F}}(W(f))(\xi) \varphi, \varphi\rangle\right|
e^{2\pi|x\cdot\xi^{\prime\prime}-y\cdot\xi^{\prime}|}
dx dy d\xi^{\prime} d\xi^{\prime\prime}<\infty,
\end{equation}
where $\varphi(x)=e^{-\pi |x|^2}$ is the Gaussian. Then $f=0.$
\end{theorem}
\begin{proof}
Let $g(x,y)=f(x,y)\langle\rho(x,y) \varphi, \varphi\rangle,$
then $\langle\mathcal{\tilde{F}}(W(f))(\xi) \varphi, \varphi\rangle$
is the Fourier transform of $g(x,y)$ at $(-\xi^{\prime\prime}, \xi^\prime).$
Now from equation (\ref{exp713}) we get
\[\int_{\mathbb{R}^{4}} \int_{\mathbb{R}^{4}}\left|g(x, y)\right|
|\mathcal{F}(g)(-\xi^{\prime\prime}, \xi^\prime)|
e^{2\pi|x\cdot\xi^{\prime\prime}-y\cdot\xi^{\prime}|}
dx dy d\xi^{\prime} d\xi^{\prime\prime}<\infty.\]
By applying the Beurling's theorem \cite{EF} for the quaternion Fourier transform,
we get $g(x,y)=0.$ Since $\langle\rho(x,y) \varphi, \varphi\rangle=e^{-\frac{\pi}{2}(|x|^2+|y|^2)}$ is
not vanishing everywhere implies $f=0.$
\end{proof}

\subsection{A remark about the set of injectivity for the quaternion twisted spherical means}
Let $\mu_r^{(k)}$ be the normalized surface measure on the sphere $S_r^{k-1}$ center at origin and radius
$r$ in $\mathbb{R}^k.$ Consider $\mathbf{\Gamma}\subseteq \mathbb{R}^k$ and
$\mathrm{G}\subseteq L_{\text{loc}}^1(\mathbb{R}^k).$ Then $\mathbf{\Gamma}$
is a set of injectivity for the spherical means in $\mathrm{G}$ if for every $g\in\mathrm{G},$
$g\ast \mu_r^{(k)}(x)=\int_{|y|=r}g(x-y)d\mu_r^{(k)}(y)=0$ for all $r>0$ and $x\in\mathbf{\Gamma}$ implies $g=0.$
The concept of injectivity is extended for the Heisenberg group in terms of the twisted spherical means (TSMs).
A subset $\Gamma\subseteq \mathbb{C}^n$ is a set of injectivity for the TSMs in
$\mathcal{G}\subseteq L_{\text{loc}}^1(\mathbb{C}^n)$ if for every
$g\in\mathcal{G},$ $g\times \mu_r^{(2n)}(z)=\int_{|w|=r}g(z-w)e^{\pi i \, \text{Im}(z.\bar{w})}d\mu_r^{(2n)}(w)=0$
for all $r>0$ and $z\in\Gamma$ implies $g=0.$ A considerable amount of work has been done in
this direction. See, e.g., \cite{NT1,Sri2,Sri4,V1}.

\smallskip

Now, consider $\Lambda\subseteq \mathbb{R}^4$ and $\mathcal{G}_\mathbb{Q}=\mathcal{G}.\left(\frac{1+k}{2}\right)
+\mathcal{G}.\left(\frac{1-k}{2}\right)\subseteq L_{\text{loc}}^1(\mathbb{R}^4,\mathbb Q).$
We say $\Lambda$ is a set of injectivity for the quaternion TSMs in $\mathcal{G}_\mathbb{Q}$ if for
every $f\in\mathcal{G}_\mathbb{Q},$
$$f\times_\mathbb Q \mu_r^{(4)}(p,q)=\int_{|(u,v)|=r}e^{\pi i(u_1q_1-p_1v_1)}
f(p-u,q-v)e^{\pi j(u_2q_2-p_2v_2)}d\mu_r^{(4)}(u,v)=0$$
for all $r>0$ and $(p,q)\in\Lambda$ implies $f=0.$ {\em In view of $(\ref{exp520}),$ it can
be concluded that $\Lambda$ is a set of injectivity for the quaternion TSMs in $\mathcal{G}_\mathbb{Q}$
if and only if $\Lambda$ and $\tilde{\Lambda}$ are set of injectivity for the TSMs in $\mathcal{G},$
where $\tilde{\Lambda}=\{(p_1,p_2,q_1,q_2):(p_1,-p_2,q_1,-q_2)\in \Lambda\}.$}

\smallskip

\noindent\textit{Helgason's support theorem.}
In a remarkable result, Helgason proved the following support theorem (see \cite{H}).
If $g$ is a continuous function on $\mathbb R^k,(k\geq2)$ such that $|x|^lg(x)$ is
bounded for each $l\in\mathbb{Z}_+,$ then $g$ is supported in the ball $B_r(0)$ if
and only if $g\ast\mu_s^{(k)}(x)=0, \forall~x\in\mathbb R^k$ and $\forall s>|x|+r.$
Together with different analogues in different setups, this is extended for the Heisenberg
group in terms of the TSM. See, e.g., \cite{DGS,EK,NT2,RS}. In view of the above discussion, we
have the following support theorem in terms of the quaternion TSM.

{\em Let $f$ be a quaternion valued continuous function on $\mathbb R^4$ such that $|(p,q)|^lf(p,q)$
is bounded for each $l\in\mathbb{Z}_+.$ Then $f$ is supported in the ball $B_r(0)$ if and only if
$f\times_\mathbb Q \mu_s^{(4)}(p,q)=0, \forall~(p,q)\in\mathbb R^4$ and $\forall s>|(p,q)|+r.$}

\bigskip

\noindent{\bf Acknowledgements:} The first named author would like to gratefully acknowledge the support
provided by IIT Guwahati, Government of India. The second named author is supported by Post Doctoral
fellowship from TIFR Center for Applicable Mathematics.

\end{document}